\documentclass[11pt,letterpaper, reqno]{amsart}
\usepackage{graphicx} 
\usepackage[table]{xcolor}
\usepackage{tikz}
\usepackage{nicematrix}
\usepackage{youngtab}
\usepackage{hyperref, cleveref}
\usetikzlibrary{arrows,matrix,positioning,fit}
\pgfdeclarelayer{background}
\pgfsetlayers{background,main}
\usepackage{subcaption}

\newtheorem{theorem}{Theorem}[section]
\newtheorem{introthm}{Theorem}[section]
 
\newtheorem{lemma}[theorem]{Lemma}
\newtheorem{proposition}[theorem]{Proposition}
\newtheorem{corollary}[theorem]{Corollary}

\newtheorem{problem}[theorem]{Problem}
\newtheorem{question}[theorem]{Question}
\newtheorem*{theorem*}{Theorem}

\usepackage{enumitem}
\newlist{steps}{enumerate}{1}
\setlist[steps, 1]{label = Step \arabic*:}

\theoremstyle{definition}
\newtheorem{definition}[theorem]{Definition}

\theoremstyle{remark}

\newtheorem{remark}[theorem]{Remark}
\newtheorem{example}[theorem]{Example}

\newcommand{\Bc}{\mathcal{B}}

\DeclareMathOperator{\reg}{reg}
\DeclareMathOperator{\bc}{bc}
\DeclareMathOperator{\sgn}{sgn}
\DeclareMathOperator{\brank}{brank}

\DeclareMathOperator{\indmatch}{indmatch}
\DeclareMathOperator{\fc}{fc}
\DeclareMathOperator{\cochord}{co-chord}

\numberwithin{equation}{section}

\title{Boolean Matrix Rank via Monomial Ideals}
\author[]{Juliann Geraci}
\address{Concordia College, Department of Mathematics,
901 8th Street S., Moorhead, MN 56562, United States}
\email{jgeraci@cord.edu}
\author[]{Alexander B. Kunin}
\address{Creighton University,  Hixson-Lied 504, Omaha, Nebraska 68178, United States}
\email{alexkunin@creighton.edu}
\author[]{Alexandra Seceleanu}
\address{University of Nebraska--Lincoln, 203 Avery Hall, Lincoln, NE 68588, United States}
\email{aseceleanu@unl.edu}
\thanks{The last author was partially supported by NSF grant DMS--2401482.}

\date{\today}

\begin{document}
\begin{abstract}
Boolean matrix factorization (BMF) has many applications in data mining, bioinformatics, and network analysis. The goal of BMF is to decompose a given binary matrix as the Boolean product of two smaller binary matrices, revealing underlying structure in the data. When interpreting a binary matrix as the biadjacency matrix of a bipartite graph, BMF is equivalent to the NP-hard biclique cover problem.

By approaching this problem through the lens of commutative algebra, we utilize algebraic structures and techniques--particularly the Castelnuovo-Mumford regularity of combinatorially defined ideals--to establish new lower bounds for Boolean matrix rank. 
\end{abstract}
\maketitle

\section{Introduction}

We introduce Boolean matrix factorization with a motivating example. Consider course enrollment in the Mathematics Department represented as a binary matrix. The rows of the matrix indicate which students enroll in which courses. Naturally, the courses are divided into specialty areas. Student X, interested in homological algebra, has to take the courses $\{$Commutative Algebra, Algebraic Topology$\}$ and Student Y, interested in combinatorial commutative algebra, has to take $\{$Commutative Algebra, Discrete Mathematics$\}$. Student Z, who is interested in combining homological algebra and combinatorial commutative algebra, should take $\{$Commutative Algebra, Algebraic Topology, Discrete Mathematics$\}$. The main point here is that student Z should only take Commutative Algebra once, thus the set union operation is more appropriate for describing the data from basis vectors as opposed to addition. Thus, Boolean operations are the right ones to consider when analyzing the data.

Once data is encoded as a Boolean matrix, Boolean Matrix Factorization (BMF) is a technique used to decompose binary matrices into smaller binary matrices, revealing hidden patterns within the data. It has found wide applications in fields such as recommender systems, bioinformatics, and computer vision. However, due to its computational complexity \cite{ORLIN1977406}, solving BMF efficiently for large datasets remains a significant challenge. To address this, various approximation methods have been developed, broadly categorized into combinatorial optimization-based approaches and continuous optimization algorithms. While continuous optimization methods relax the binary constraint and optimize the matrices in a continuous space, this paper focuses on algebraic methods inspired by the combinatorial optimization-based techniques.

Combinatorial optimization approaches to BMF leverage discrete structures to develop efficient approximations.  Notable methods such as \\ GRECOND+\cite{BELOHLAVEK}, PANDA+\cite{Lucchese}, and ASSO \cite{Miettinen08} each employ distinct strategies to reduce reconstruction error while managing computational cost. These methods offer significant improvements over exact algorithms, which struggle with scalability for large datasets. By utilizing heuristics and iterative refinement, these algorithms strike a balance between factorization quality and computational cost.

Boolean rank encodes the optimal result of BMF, see Definition \ref{def:brank}. In this paper we introduce two algebraic lower bounds for the Boolean rank of a matrix, both expressed in terms of Castelnuovo-Mumford regularity -- an invariant measuring homological complexity in the theory of graded rings. Along the way, we associate to a Boolean matrix various algebraic and combinatorial structures: a bipartite graph, its edge ideal, a new simplicial complex termed the \textit{isolation complex} and its corresponding monomial ideal dubbed the \textit{isolation ideal}. These structures provide new insights into Boolean matrix properties.

It is well-known that the Boolean rank of a binary matrix corresponds to the biclique cover number of its associated bipartite graph. Based on  this interpretation, we prove  in section \ref{s: biclique} that the Castelnuovo-Mumford regularity of the quotient by the edge ideal of a bipartite graph associated with a Boolean matrix gives a lower bound on its Boolean rank. This can also be deduced by appropriately interpreting the work of Woodroofe \cite{Woodroofe}.

\begin{introthm}[Theorem \ref{thm: reg}]
For a bipartite graph $G$ with edge ideal $I_G$ and biadjacency matrix $A$, the regularity of the quotient ring of $I_G$ gives a lower bound on the Boolean rank of $A$, namely 
\[
\reg(R/I_G)\leq \brank(A).
\]

\end{introthm}

 The isolation number of a Boolean matrix-- the size of the largest set of isolated ones in the matrix-- also serves as a lower bound for Boolean rank. In section \ref{s: isolated ones} we construct a monomial ideal, which we call the \textit{isolation ideal}, whose Castelnuovo-Mumford regularity recovers the isolation number of a Boolean matrix. This yields a second lower bound on Boolean rank.

 \begin{introthm}[Theorem \ref{thm: isolation number = reg}]
Let $A$ be a Boolean matrix. Then the regularity of the quotient ring $k[A]/J_A$ of the isolation ideal $J_A$ gives a lower bound on the Boolean rank of $A$, namely
\[
\reg(k[A]/J_A)\leq \brank(A).
\]
\end{introthm}
 
  The paper concludes with section \ref{s: computations} detailing a computational analysis of several matrix structures (block matrices, overlapping blocks, and identity complement matrices), highlighting scenarios where our algebraic bounds closely estimate the Boolean rank and cases where they diverge significantly. 

\section{Preliminaries}
This section provides foundational definitions in commutative algebra and Boolean matrix factorization, both of which are central to understanding the theoretical framework of the paper. Commutative algebra deals with the properties of commutative rings and their modules, offering a rich structure for analyzing algebraic systems.
\subsection{The Boolean Matrix Factorization Problem}
Before stating the problem, we introduce the algebraic structure underlying it.

\begin{definition}
The \textit{Booleans}, $\Bc$, are defined to be the semiring on $\{0,1\}$ equipped with operations $\vee$ (logical ``or") and $\wedge$ (logical ``and") specified in the tables below:
\[
\begin{tabular}{c|cc}
    $\vee$ & 0 & 1 \\
    \hline
    0 & 0 & 1\\
    1 & 1 & 1
\end{tabular}
\qquad
\begin{tabular}{c|cc}
    $\wedge$ & 0 & 1 \\
    \hline
    0 & 0 & 0\\
    1 & 0 & 1
\end{tabular}
\]
\end{definition}

Note that $\Bc$ is \textit{not} a field, or even a ring, as there is no additive inverse. This fact makes the notion of Boolean rank as well as the related problem of finding matrix factorizations for matrices with entries in $\Bc$ difficult as many standard matrix factorization techniques fail to work outside the realm of rings and fields.


Denote by $\Bc^{m,n}$ the set of $m\times n$ matrices with entries in $\Bc$. For matrices $V\in \Bc^{m,r}, H\in \Bc^{r,n}$ we write $V\wedge H$ to mean Boolean matrix multiplication. This is defined by analogy with matrix multiplication over a ring, with addition replaced by $\vee$ and multiplication by $\wedge$, that is,
\[
[V\wedge H]_{ij}=\bigvee_{\ell=1}^r V_{i\ell}\wedge H_{\ell j}.
\]

The goal of matrix decomposition is to factorize an input matrix into two smaller factor matrices, whose product reconstructs the original matrix. Here we give a precise statement of the Boolean Matrix Factorization (BMF) problem.

\begin{problem}
Given a matrix $A\in \Bc^{m,n}$, find matrices $V\in \Bc^{m,r}$ and $H\in \Bc^{r,n}$ such that $A = V\wedge H$ and $r$ is as small as possible.
\end{problem}

The answer to this problem is the Boolean rank.

\begin{definition}\label{def:brank}
    The \textit{Boolean rank} of an $m\times n$ binary matrix $A\in \Bc^{m,n}$ is 
    \[
    \brank(A)=\text{min}\{r \mid A = V\wedge H,V\in \Bc^{m,r}, H\in \Bc^{r,n} \}.
    \]
\end{definition}

Determining the Boolean rank is computationally challenging; however, significant effort has been devoted to approximating the factorization \cite{RecentDevs} \cite{DeSantis21}, and consequently, the rank itself.

\subsection{Commutative Algebra Background}
Let $R=k[x_1,\dots,x_n]$ denote a polynomial ring in $n$ variables, with $k$ a field. 
\begin{definition}
A \textit{free resolution} of a finitely generated $R$-module $M$ is a sequence of homomorphisms of $R$-modules
\[
F_\bullet = 0 \to F_r \xrightarrow{d_{r}} \ldots \xrightarrow{d_1} F_0 \xrightarrow{d_0} M \to 0
\]
such that $F_\bullet$ is exact, each $F_i$ is a finitely generated free $R$-module and  $M \cong F_0/\text{Im}(d_1)$.

The ring $R$ and the $R$-module $M$  are \textit{$\mathbb{N}$-graded} if there exist decompositions $R= \bigoplus_{i=0}^\infty R_i$  and  $M= \bigoplus_{i=0}^\infty M_i$ with $R_mR_n\subseteq R_{m+n}$ and  $R_mM_n\subseteq M_{m+n}$ for all $m,n\geq 0$. 
If $R$ and $M$ are graded, and the map $d_i$ preserves degrees, we call $F_\bullet$ a \textit{graded free resolution.}

A free resolution $F_\bullet$ is called a \textit{minimal free resolution} if 
\[
d_{i+1}(F_{i+1})\subseteq (x_1,\ldots,x_n)F_i \text{ for all } i\geq 0.
 \]
\end{definition}
This means that the matrices representing the differential maps $d_i$ with respect to any homogeneous choice of bases contain no constant entries other than possibly zero.
\begin{definition}
Let $M$ be a graded $R$-module with minimal free resolution
\[
F_\bullet = 0 \to F_r \xrightarrow{d_{r}} \ldots \xrightarrow{d_1} F_0 \xrightarrow{d_0} M \to 0.
\]
The number of elements of degree $j$ in  any homogeneous basis of $F_i$, denoted $\beta^R_{i,j}(M)$, is called the {\em $i$-th graded Betti number of $M$ in  degree $j$}.
\end{definition}
We would like to know the degrees where the non-zero Betti numbers are located. While knowing that precise information is often not possible, it is useful to have an upper bound for degrees where the Betti numbers are non-zero. Such bounds are provided by the notion of \textit{regularity}.
\begin{definition}
The \textit{Castelnuovo-Mumford regularity} (or just \textit{regularity}) of  $M$ over $R$, denoted $\reg_R(M)$, is defined as
\[
\reg_R(M)=\text{max}\{j\mid \beta_{i,i+j}^R(M)\neq 0 \text{ for some $i$}\}.
\]
We often drop the subscript $R$, writing $\reg(M)$ when the ring $R$ is evident.
\end{definition}  
Since the regularity measures the degrees of elements involved in constructing a free resolution of $M$, it can be viewed as a measure for the complexity of computing such a free resolution, often referred to as the homological complexity of $M$. 

It follows from the definition that for a graded ideal $I$ of $R$ we have $\reg_R(I)=\reg_R(R/I)+1$, giving a translation between complexity of $I$ and $R/I$.

In this paper we study a special class of ideals called \textit{monomial ideals}. 
\begin{definition}
    An ideal $I$ of a polynomial ring is a \textit{monomial ideal} if it can be generated by monomials, that is $I = (f_1,\ldots,f_k)$, where the generators, $f_i$, are monomials for all $i$.
\end{definition}

Within the class of monomial ideals and their quotients, regularity behaves in a  subadditive manner, as shown by the following result.

\begin{proposition}\label{thm: reg sum}
    Let $I$ and $J$ be monomial ideals. Then
    \[
    \reg(R/(I+J))\leq \reg(R/I)+\reg(R/J)
    \]
\end{proposition}
\begin{proof}
    \cite[Corollary 3.2(b)]{Herzog} states that for monomial ideals, $\reg(I + J) \leq \reg(I) + \reg(J) - 1$.
    
    Using the identity $\reg(R/I)=\reg(I)-1$ for nonzero homogeneous ideals, we obtain 
    \[
    \reg(R/(I+J))= \reg(I+J)-1\leq \reg(I)+\reg(J)-2= \reg(R/I)+\reg(R/J).
    \]
\end{proof}

By induction the next statement follows.

\begin{corollary}\label{2.9}
    For a finite collection $\{I_\lambda\}$ of monomial ideals, 
    \[
    \reg\left(R/\left(\sum_{\lambda} I_\lambda\right)\right)\leq \reg(R/I_1)+\ldots + \reg(R/I_\lambda).
    \]
\end{corollary}

\subsection{Combinatorics Background}
 
Our main ideals of interest are edge ideals of graphs.

\begin{definition}
Let $G=(V,E)$ be a finite, simple graph with vertex set $V=\{x_1,\ldots, x_n\}$ and edge set $E$. We will not consider graphs with multiple edges between vertices or loops. Let $R=k[x_1,\ldots,x_n]$ be the polynomial ring over a field $k$. The \textit{edge ideal of $G$} in $R$ is defined as the ideal generated by monomials representing the edges of $G$, that is,
    \[
    I_G=\left(x_ix_j:\{x_i,x_j\} \in E\right)
    \]
\end{definition}

Castelnuovo-Mumford regularity of an edge ideal is related to the induced matching number of the graph. A matching is a set of edges in which no two edges share a vertex. An \textit{induced matching }in a graph $G$ is a matching which forms an induced subgraph of $G$. We denote by $\text{indmatch}(G)$ the maximum number of edges in any induced matching. 
\begin{lemma}\cite[Lem.\,2.2]{Katzman}\label{lem:indmatch}
For any graph $G$, we have $\reg(R/I_G)\geq \indmatch(G)$.
\end{lemma}

In this paper we will be focusing on bipartite graphs and particular subgraphs thereof.

\begin{definition}\label{def:biclique}
A \textit{bipartite graph} is a graph whose vertices can be partitioned into sets $X$ and  $Y$  so that every edge connects a vertex in $X$ to one in $Y$.  We denote a bipartite graph with vertex sets $X$ and $Y$ as $G=(X\sqcup Y,E)$.

\noindent A \textit{biclique} is a complete bipartite graph.

\noindent A \textit{biclique cover} of a bipartite graph $G$ is a collection of bicliques $\{C_i =(X_i\cup Y_i, E_i); i = 1,\ldots, t\}$ with each $X_i\subseteq X$, $Y_i\subseteq Y$ and $E_i\subseteq E$ such that $\bigcup_{i=1}^t E_i = E$.

\noindent The \textit{biclique cover number} of a graph $G$, denoted $\text{bc}(G)$, is the smallest number of bicliques in any biclique cover of $G$.  

\end{definition}

\begin{definition}
Let $G=(X\sqcup Y,E)$ be a bipartite graph with vertex sets $X=\{x_1,\ldots,x_m\}$ and $Y= \{y_1,\ldots,y_n\}$.  The \textit{biadjacency matrix} is the $m\times n$ matrix $A(G)$ in which $a_{i,j} = 1$ if $(x_i, y_j) \in E$ and $a_{ij}=0$ otherwise.
\end{definition}


\begin{example}
We illustrate two different ways to represent the same data, coming from a bipartite graph. 
\begin{figure}[h!]
\centering
\captionsetup[subfigure]{labelformat=empty}
\begin{subfigure}[b]{0.3\textwidth}
\centering
\begin{tikzpicture}
    \node[shape=circle,scale=0.5, draw=black, fill=black,label=left:$x_3$] (A) at (0,0) {};
    \node[shape=circle,scale=0.5,draw=black, fill=black,label=left:$x_2$] (B) at (0,.5) {};
    \node[shape=circle,scale=0.5,draw=black, fill=black,label=left:$x_1$] (C) at (0,1) {};
    \node[shape=circle,scale=0.5,draw=black, fill=black,label=right:$y_4$] (D) at (1,0) {};
    \node[shape=circle,scale=0.5,draw=black, fill=black,label=right:$y_3$] (E) at (1,.33) {};
    \node[shape=circle,scale=0.5,draw=black, fill=black,label=right:$y_2$] (F) at (1,.66) {};
    \node[shape=circle,scale=0.5,draw=black, fill=black,label=right:$y_1$] (G) at (1,1) {};
    \path (C) [black] edge node[left] {} (D);
    \path (C) [black] edge node[left] {} (E);
    \path (C) [black] edge node[left] {} (F);
    \path (C) [black] edge node[left] {} (G);
    \path (B) [black] edge node[left] {} (F);
    \path (B) [black] edge node[left] {} (G);
    \path (A) [black] edge node[left] {} (G);
\end{tikzpicture}
\subcaption[]{G}
\end{subfigure}
\begin{subfigure}[b]{0.3\textwidth}
\centering
$\begin{bmatrix}
1&1&1&1\\
1&1&0&0\\
1&0&0&0
\end{bmatrix}$
\subcaption[]{A(G)}
\end{subfigure}
\caption{A bipartite graph $G$ along with its biadjacency matrix $A(G)$. The edge ideal of $G$ is $I_G =\langle x_1y_1,x_1y_2,x_1y_3,x_1y_4,x_2y_1,x_2y_2,x_3y_1\rangle.$ }
\end{figure}
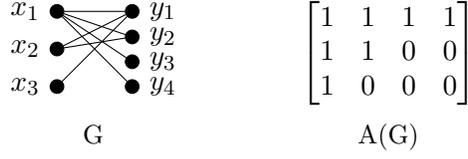

\end{example}


\section{Boolean Rank via Biclique Covers}\label{s: biclique}

Any binary matrix, $A\in \Bc^{m\times n}$, can be viewed as the biadjacency matrix of a bipartite graph $G=(X\sqcup Y, E)$ where the rows of $A$ correspond to the vertices in $X$, the columns of $A$ correspond to the vertices in $Y$, and $a_{i,j}=1$ if and only if $\{x_i,y_j\}\in E$. Conversely every bipartite graph $G$ gives rise to a biadjacency matrix $A(G)$ and this establishes a bijective correspondence between bipartite graphs and binary matrices.

For a bipartite graph $G$, the \textit{biclique cover number} of $G$ is equal to the \textit{Boolean rank} of $A(G)$; see \cite{RecentDevs}. That is,
\begin{equation}\label{eq: bc=brank}
\text{bc}(G)=\brank(A(G)).
\end{equation}

\begin{example}\label{ex: running ex}
The factorization below
\[
A=\begin{bmatrix}
1&1&0&0\\
0&1&1&0\\
0&0&1&1\\
1&1&1&0\\
0&1&1&1\\
1&1&1&1
\end{bmatrix}= \begin{bmatrix}
\textcolor{teal}{1}&\textcolor{orange}{0}&\textcolor{purple}{0}\\
\textcolor{teal}{0}&\textcolor{orange}{1}&\textcolor{purple}{0}\\
\textcolor{teal}{0}&\textcolor{orange}{0}&\textcolor{purple}{1}\\
\textcolor{teal}{1}&\textcolor{orange}{1}&\textcolor{purple}{0}\\
\textcolor{teal}{0}&\textcolor{orange}{1}&\textcolor{purple}{1}\\
\textcolor{teal}{1}&\textcolor{orange}{1}&\textcolor{purple}{1}
\end{bmatrix} \wedge \begin{bmatrix}
\textcolor{teal}{1}&\textcolor{teal}{1}&\textcolor{teal}{0}&\textcolor{teal}{0}\\
\textcolor{orange}{0}&\textcolor{orange}{1}&\textcolor{orange}{1}&\textcolor{orange}{0}\\
\textcolor{purple}{0}&\textcolor{purple}{0}&\textcolor{purple}{1}&\textcolor{purple}{1}
\end{bmatrix}= V \wedge H.
\]
can be visualized using bipartite graphs. Taking $A$ as the biadjacency matrix of the graph $G$ with $V(G)=X\sqcup Y$, then each column of the first factor matrix, $V$, and  corresponding row in the second factor matrix, $H$, determine a biclique in the biclique cover of $G$. If the entry in the $i$th column of $V$ is a one, that indicates that the corresponding vertex of $X$  is a member of the $i$th biclique. If the entry in the $i$th row of $H$ is a one, that indicates that the corresponding vertex of $Y$ is a member of the $i$th biclique.
\begin{center}
    \includegraphics[scale=0.18]{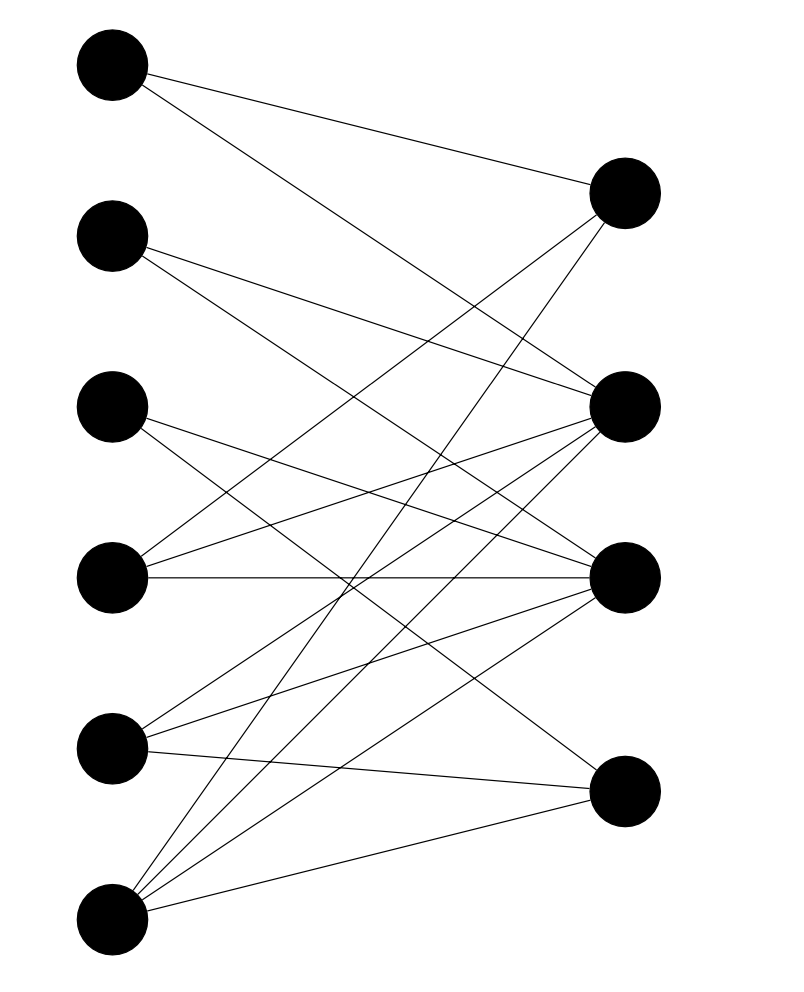} \raisebox{12ex}{=}
    \includegraphics[scale=0.25]{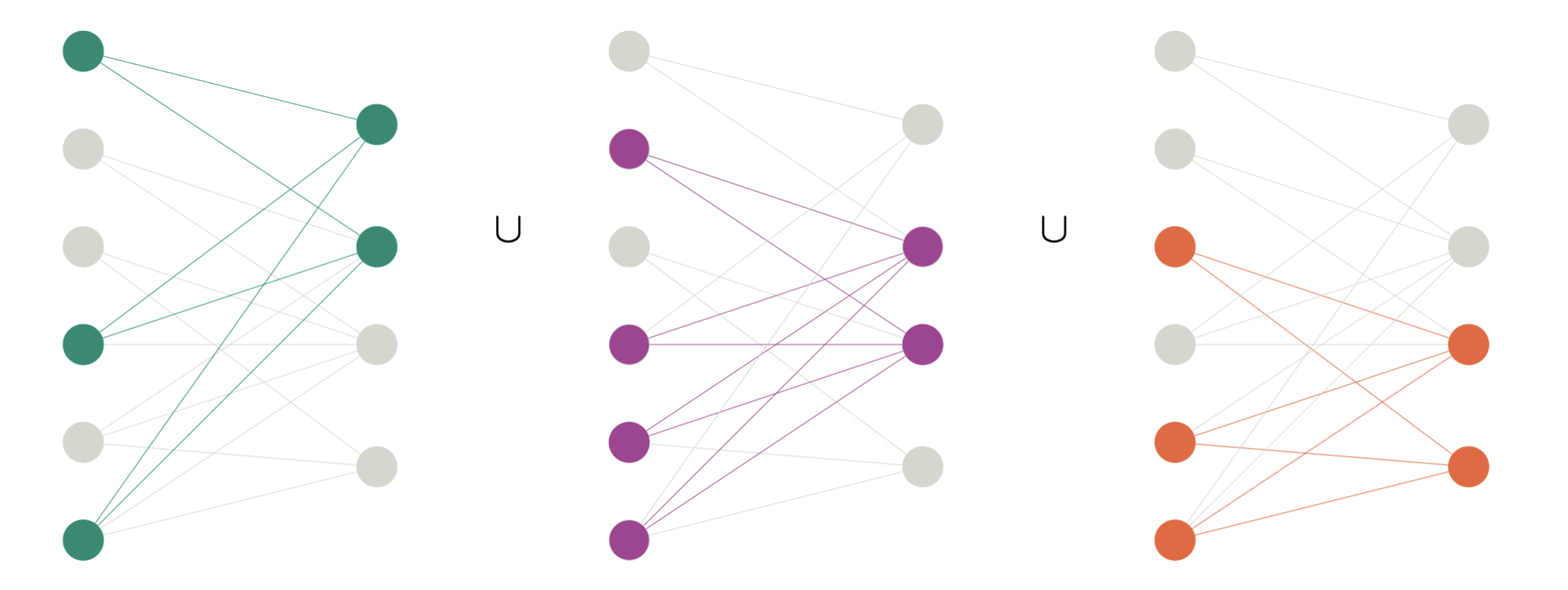}
\end{center}
The biclique cover pictured above shows that the Boolean rank of $A$ is at most three. It can be verified that there is no way to cover $G$ by two bicliques, thus $\brank(A)=3$ by \eqref{eq: bc=brank}.
\end{example}

There has been much theory developed relating algebraic invariants of the edge ideal of a graph $G$ to combinatorial invariants of $G$ \cite{HaWoodroofe}. A natural question one might ask is:

\begin{question}
Do algebraic invariants of the edge ideal of a bipartite graph $G$ give insight into the Boolean rank of $A(G)$?
\end{question}

We answer this question by demonstrating a lower bound on the Boolean rank of a matrix in terms of the Castelnuovo-Mumford regularity of a related edge ideal. This hinges on the observation that the regularity of the quotient ring of an edge ideal of a biclique is equal to one.


\begin{lemma}\cite[Cor.\,2.2]{NagelCorso}\label{lem: regularity biclique}
If $C$ is a biclique with edge ideal $I_C$, then
\[
\reg(R/I_C)=1.
\]
\end{lemma}


The following theorem can be obtained as a corollary of \cite[Theorem 1]{Woodroofe}, which states $\reg(R/I_G) \leq \cochord(G)$, the number of \emph{co-chordal} graphs needed to cover $G$. The complement of a biclique is two disjoint cliques, hence chordal, so we immediately obtain 
\[
\reg(I/I_G) \leq \cochord(G) \leq \bc(G) = \brank(A(G)).
\]

We include a proof of the theorem using \Cref{lem: regularity biclique} to bound Boolean rank in order to illustrate a technique which will reoccur in later sections. 

\begin{theorem}\label{thm: reg}
Let $G$ be a bipartite graph with biadjacency matrix $A$. Then 
\[
\reg(R/I_G)\leq \brank(A).
\]
\end{theorem}
\begin{proof}
Let $\brank(A)=r$. By \eqref{eq: bc=brank}, $\text{bc}(G)=\brank(A)$, so there exists a covering $G=\bigcup_{\ell=1}^r C_\ell$ of $G$ by bicliques $C_\ell$ that yields the following relationship between the associated edge ideals $I_G=\sum_{\ell=1}^r I_{C_\ell}$.

 By Proposition \ref{thm: reg sum} it follows that 
\begin{equation}\label{eq: 1 reg}
\reg(R/I_G) = \reg\left (R\bigg/\sum_{\ell=1}^r I_{C_\ell} \right ) \leq \sum_{\ell=1}^r \reg(R/I_{C_\ell}).
\end{equation}

Lemma \ref{lem: regularity biclique} gives $\reg(R/I_C)=1$ when $C$ is a biclique. Hence we have the equality 
\begin{equation}\label{eq: 2 reg}
\sum_{\ell=1}^r \reg(R/I_{C_\ell})= r= \text{bc}(G)=\brank(A). 
\end{equation}
Therefore, from \eqref{eq: 1 reg} and \eqref{eq: 2 reg} it follows that
\[
\reg(R/I_G)\leq \brank(A). \qedhere
\]
\end{proof}


Lemma \ref{lem:indmatch} along with the previously stated result give rise to a very useful chain of inequalities for a bipartite graph $G$ with biadjacency matrix $A$.

\begin{corollary}
Let $G$ be a bipartite graph with biadjacency matrix $A$. Then 
    \begin{equation}\label{eq: inequalities}
    \indmatch(G) \leq \reg(R/I_G) \leq  \brank(A) = \bc(G).
\end{equation}
\end{corollary}

\begin{example} \label{ex: running example reg}
Returning to Example \ref{ex: running ex} of 
\[
A=\begin{bmatrix}
1&1&0&0\\
0&1&1&0\\
0&0&1&1\\
1&1&1&0\\
0&1&1&1\\
1&1&1&1
\end{bmatrix},
\]
one computes $\reg(R/I_G)=2$, which is indeed less than the Boolean rank of $A$ which is 3, see Example \ref{ex: running ex}.
\end{example}

\begin{remark}
Note that $\reg(R/I_G)$ can be arbitrarily far from the Boolean rank. See Example \ref{ex: anti id example} for a case study.
\end{remark}

\section{Boolean Rank via Isolated Ones}\label{s: isolated ones}
\subsection{Isolated Ones}

We now consider a lower bound on the biclique cover number based on obstructions to bicliques. That is, if the graph pictured below is an induced subgraph of some bipartite graph $G$, then edges $\{i,j\}$ and $\{k,\ell\}$ must belong to different bicliques in a biclique cover of $G$ since there is no edge $\{i,\ell\}$. This motivates our definition of the isolation number and isolation complex of a matrix $A.$\\

\begin{center}
    \begin{tikzpicture}
    \node[shape=circle,scale=0.5, draw=black, fill=black,label=left:$k$] (A) at (0,0) {};
    \node[shape=circle,scale=0.5,draw=black, fill=black,label=left:$i$] (B) at (0,.5) {};
    \node[shape=circle,scale=0.5,draw=black, fill=black,label=right:$\ell$] (D) at (1,0) {};
    \node[shape=circle,scale=0.5,draw=black, fill=black,label=right:$j$] (F) at (1,.5) {};
    \path (A) [black] edge node[left] {} (D);
    \path (F) [black] edge node[left] {} (A);
    \path (F) [black] edge node[left] {} (B);
\end{tikzpicture}\end{center}

\begin{definition}\label{def: isolated}
Let $A\in \Bc^{m,n}$. A pair of entries $\{a_{ij},a_{k\ell}\}$ is called \textit{isolated} if  $a_{ij}=a_{k\ell}=1$ and $a_{i\ell}\wedge a_{kj}=0$.  A subset $T$ of the entries of $A$ which have value $1$ is an \textit{isolated set} if it is of size one or all pairs of elements in $T$ are isolated. We say that $a_{ij}$ is \textit{isolated from} $a_{k\ell}$ if $\{a_{ij},a_{k\ell}\}$ is an isolated pair.\\

\noindent The \textit{isolation number} of a Boolean matrix $A$, $\iota(A)$, is the maximum of the sizes of all isolated sets in $A$. 
\end{definition}

Note that isolated pairs cannot exist in the same row or column. For example, say $a_{ij}=a_{kj}=1$ are in the same column. Then $a_{ij}\wedge a_{kj}=1\neq 0$ so these entries cannot form an isolated pair. Similarly if $a_{ij}=a_{ik}=1$ are in the same row they cannot form an isolated pair. 

\begin{example}
    If 
        \[
        A = \begin{bmatrix} 1&0&1\\0&1&0\\0&0&1\end{bmatrix}
        \]
    then $\{a_{11},a_{22}\},\{a_{22},a_{33}\},\{a_{11},a_{33}\},\{a_{13},a_{22}\}$ are isolated pairs. This means that $\iota(A)=3$ since $\{a_{11},a_{22},a_{33}\}$ is the largest set of pairwise isolated ones. Note that $a_{13}$ cannot be added to this set since $a_{13}$ is not isolated from $a_{11}$ nor $a_{33}$ by the preceding remarks.  
\end{example}

The isolation number is a lower bound on the Boolean rank.

\begin{proposition}\cite[Prop.\,2]{DeSantis21}\label{prop: iso-brank}
    Let $A\in \Bc^{m,n}$. Then
    \[
    \iota(A) \leq \brank(A).
    \]
\end{proposition}

\begin{example}
Let $A$ be the same matrix as in \Cref{ex: running example reg}. Highlighted below is one (of the many) isolated sets of size $3$. By inspection, we cannot find a set of pairwise isolated ones of size 4. Thus, $\iota(A)=3$.
\[
A=\begin{bmatrix}
\textcolor{magenta}{\bf 1}&1&0&0\\
0&1&1&0\\
0&0&1&1\\
1&1&\textcolor{magenta}{\bf 1}&0\\
0&1&1&\textcolor{magenta}{\bf 1}\\
1&1&1&1
\end{bmatrix}.
\]
\end{example}
\begin{remark}
Note that $\iota(A)$ can also be arbitrarily far from the Boolean rank. See Example \ref{ex: anti id example}.
\end{remark}

\subsection{Algebraic Computation of the Isolation Number}
Our goal is to translate the isolation number into an algebraic invariant of the matrix $A$. We observe that isolated sets are closed under taking subsets, and so the collection of all isolated sets forms a simplicial complex. The Stanley-Reisner correspondence defines a monomial ideal which can be obtained from a simplicial complex. We show in \Cref{thm: isolation number = reg} that the Castelnuovo-Mumford regularity of this Stanley-Reisner ring modulo the squares of the variables is the isolation number of the matrix. 

\begin{definition}\label{def: Delta}
Let $\Delta_A$ be the simplicial complex whose faces are the isolated sets of the matrix $A$. Since each entry of $A$ that is a 1 is an isolated set of size one, these will be the vertices of $\Delta_A$.
\end{definition}

\begin{remark}
 Note that   $\Delta_A$ is a flag complex, meaning it satisfies the property that  for every subset $\sigma$ of at least two vertices of $\Delta_A$, if every pair of vertices in $\sigma$ is a face of the complex, then $\sigma$ itself is a face of  the complex too. This follows from \Cref{def: Delta} and \Cref{def: isolated}.  
\end{remark}

We proceed to construct a series of ideals associated to $\Delta_A$. To do so we replace the nonzero entries of $A$ with distinct variables.

\begin{definition}\label{def: A}
Let $A\in\Bc^{m\times n}$. Define the matrix $A[x]$ as follows 
\[
A[x]_{i,j}=\begin{cases} x_{ij}& a_{ij}=1\\ 0 & a_{ij}=0 \end{cases}.
\]
\end{definition} 
\begin{example}
Let $A$ be as in Example 4.2. Then $\Delta_A$ with labeled vertices coming from $A[x]$ can be realized as 
\begin{center}
\tikzset{every picture/.style={line width=0.75pt}} 
\begin{tikzpicture}
    \node[circle, fill=black, inner sep=1.5pt, label=below:$x_{22}$] (v0) at (0,0) {};
    \node[circle, fill=black, inner sep=1.5pt, label=below:$x_{33}$] (v1) at (2,0) {};
    \node[circle, fill=black, inner sep=1.5pt, label=above:$x_{11}$] (v2) at (0,2) {};
    \node[circle, fill=black, inner sep=1.5pt, label=above:$x_{13}$] (v3) at (2,2) {};

    \fill[gray!20, opacity=1] (v0.center) -- (v1.center) -- (v2.center) -- cycle;

    \draw[thick] (v0) -- (v2);
    \draw[thick] (v0) -- (v1);
    \draw[thick] (v1) -- (v2);
    \draw[thick] (v1) -- (v3);
\end{tikzpicture}

\end{center}
\end{example}

\begin{definition}
Let $k[A]= k[x_{ij} \mid a_{ij}=1]$ be the polynomial ring with indeterminates corresponding to the nonzero entries of $A[x]$. Denote by $I_{\Delta_A}\subset k[A]$  the {\em Stanley-Reisner ideal} of $\Delta_A$
\begin{equation}\label{eq: SR ideal}
I_{\Delta_A}=\left( \prod_{a_{ij}\in S} x_{ij} \mid S \not \in \Delta\right).
\end{equation}
This ideal is generated by monomials corresponding to non-faces of $\Delta_A$, that is the \textit{non-isolated} sets of $A$.
\end{definition}

We want to find an explicit list of generators of $I_{\Delta_A}$. It is well known that the generators of the Stanley-Reisner ideal of a flag complex are quadratic. Ultimately, we will identify them as the $2\times 2$ minors of the matrix $A[x]$ introduced in Definition \ref{def: A} that are monomials. To do this, we must introduce a few new notions.

\begin{definition}
    Let $A$ be an $m\times n$ Boolean matrix. Let $I_t(A)\subseteq k[A]$ denote the \textit{$t$-th determinantal ideal} of $A[x]$; that is, the ideal generated by all $t\times t$ minors of $A[x]$.
\end{definition}


\begin{lemma}\label{lem: 2x2 minors}
    Let $A\in \Bc^{m,n}$. Entries $a_{ij}$ and $a_{k\ell}$ are an isolated pair if and only if $x_{ij}x_{k\ell}\in I_2(A)$. 
\end{lemma}
\begin{proof}
Suppose $a_{ij}$ and $a_{k\ell}$ are an isolated pair, 
 then we have $x_{ij}\neq 0, x_{k\ell}\neq 0$ and since $a_{i\ell}\wedge a_{kj}=0$ there are three options for the $2\times 2$ submatrix of $A[x]$ containing these entries
\[
\begin{bmatrix}
    x_{ij}& 0\\
    0 & x_{k\ell}
\end{bmatrix} \text{ or } 
\begin{bmatrix}
    x_{ij}& 0\\
    x_{kj} & x_{k\ell}
\end{bmatrix} \text{ or }
\begin{bmatrix}
    x_{ij}& x_{i\ell}\\
    0 & x_{k\ell}
\end{bmatrix}.
\]
In every case the determinant is $x_{ij}x_{k\ell}$. Hence, $x_{ij}x_{k\ell}\in I_2(A)$, as desired. 

On the other hand suppose $x_{ij}x_{k\ell}\in I_2(A)$. In particular this conveys that $x_{ij},x_{k\ell}\in k[A]$ so that $a_{ij}=a_{k\ell}=1$. 
Suppose towards a  contradiction that $\{a_{ij},a_{k\ell}\}$ is not an isolated pair, 
then the following is  a $2\times 2$ submatrix of $A[x]$\[
\begin{bmatrix}
    x_{ij} & x_{i\ell}\\
    x_{kj}& x_{k\ell}
\end{bmatrix}.
\]
Since $x_{ij}x_{k\ell}\in I_2(A)$, this monomial is a linear combination of determinants of $2\times 2$ submatrices of $A[x]$ including the determinant of the submatrix shown above.
However, the term $x_{kj}x_{i\ell}$ in the determinant of the above matrix appears in no other $2\times 2$ minor of $A[x]$ and so it cannot cancel in the linear combination. This yields a contradiction. Thus it must be that $\{a_{ij},a_{k\ell}\}$ is an isolated pair.
\end{proof}

\Cref{lem: 2x2 minors} gives a necessary and sufficient criterion for  an edge to appear in $\Delta_A$ in terms of the $2\times 2$ minors of $A[x]$. For higher dimensional faces of $\Delta$, we can give a sufficient condition for a set to be isolated, and hence for the corresponding  face to appear in $\Delta_A$. However, this condition is no longer necessary, that is, we can no longer say that an isolated set guarantees a monomial generator in the determinantal ideal. 

\begin{theorem}\label{thm: monomials=isolated}
        Let $A$ be an $m\times n$ Boolean matrix. If for some integer $s\geq 2$ the monomial $x_{i_1 j_1}\cdots x_{i_sj_s}\in I_s(A)$, then  $\{a_{i_k j_k}\mid 1\leq k\leq s\}$ is an isolated set.
\end{theorem}
\begin{proof}
Proceed by induction on the degree of the monomial generator. 

\noindent \underline{Base Case:} the case $s=2$ is covered in \Cref{lem: 2x2 minors}.

\noindent \underline{Induction Hypothesis:} Assume that if $x_{i_1 j_1}\cdots x_{i_{s-1}j_{s-1}}\in I_{s-1}(A)$, then $\{a_{i_1 j_1}\cdots a_{i_{s-1}j_{s-1}}\}$ is an isolated set. Additionally assume that $s\geq 3$.

Suppose we know that $x_{i_1 j_1}\cdots x_{i_sj_s}\in I_s(A)$. This implies that the given monomial is a $k$-linear combination of $s\times s$ minors of $A[x]$. Since the sets of monomials that appear in the expressions of any two $s\times s$ minors of $A[x]$ are disjoint, it follows that there are no possible cancellations in any linear combination, including within a single determinant. Hence $x_{i_1 j_1}\cdots x_{i_sj_s}\in I_s(A)$ is in fact equal to the minor of $A[x]$ in rows $\{i_1, \ldots, i_s\}$ and columns $\{j_1, \ldots, j_s\}$. Denote this minor by $\text{det}_s$. 
We have established
\[
\text{det}_s = x_{i_1 j_1}\cdots x_{i_sj_s}.
\]
Further denote the minor of $A[x]$ in rows $\{i_1, \ldots, i_{s-1}\}$ and columns $\{j_1, \ldots, j_{s-1}\}$ by $\text{det}_{s-1}$.
In particular, this requires that each $x_{i_k,j_k} \in k[A]$ and hence $a_{i_k,j_k}=1$.

In the following considerations we write $x_{i,j}$ for the entry of $A[x]$ in row $i$ and column $j$ regardless of whether this entry is zero or a variable in $k[A]$. By the definition of determinant, we have 
\begin{align}
\text{det}_s  = \sum_{\sigma \in S_s}\sgn(\sigma)x_{i_1 j_{\sigma(1)}}\cdots x_{i_sj_{\sigma(s)}} \label{eq: 1}\\
\text{det}_{s-1}= \sum_{\sigma \in S_{s-1}}\sgn(\sigma)x_{i_1 j_{\sigma(1)}}\cdots x_{i_{s-1}j_{\sigma(s-1)}},  \label{eq:2}
\end{align}
where $\sgn(\sigma)$ denotes the signature of $\sigma.$
Multiplying \eqref{eq:2} by $x_{i_s j_s}$, we get the following expression
\[
\sum_{\sigma \in S_{s-1}}\sgn(\sigma)x_{i_1 j_{\sigma(1)}}\cdots x_{i_{s-1}j_{\sigma(s-1)}}\cdot x_{i_s j_s},
\]
which is a part of \eqref{eq: 1}. But we assumed that \eqref{eq: 1} is equal to $x_{i_1 j_1}\cdots x_{i_sj_s}$, so we must have that every term that is \textbf{not} $x_{i_1 j_1}\cdots x_{i_sj_s}$ is zero. In particular, $x_{i_1 j_{\sigma(1)}}\cdots x_{i_{s-1}j_{\sigma(s-1)}}\cdot x_{i_s j_s}=0$, unless $\sigma$ is the identity permutation. Since $k[x_{11},x_{12},\ldots, x_{ss}]$ is a domain and $x_{i_s,j_s}\neq 0$ we deduce that $x_{i_1 j_{\sigma(1)}}\cdots x_{i_{s-1}j_{\sigma(s-1)}}=0$, for all $\sigma$ that are not the identity permutation. Therefore  $\text{det}_{s-1}=x_{i_1 j_1}\cdots x_{i_{s-1}j_{s-1}}\in I_{s-1}(A)$.
The induction hypothesis now yields the following isolated set 
$T_{s}=\{a_{i_1,j_1},\ldots, a_{i_{s-1},j_{s-1}}\}$.

The same argument, applied for any $1\leq \ell \leq s$, yields that the set
\[
T_\ell=\{a_{i_1 j_1}\cdots \widehat{a_{i_\ell j_\ell}} \cdots a_{i_{s}j_{s}}\}
\]
is an isolated set. Since any pair $\{a_{i_p j_p},a_{i_qj_q}\}$  with $1\leq p \leq s$, $1\leq q\leq s$ and $q\neq p$, is contained in at least one of the sets  $T_\ell$, $\{a_{i_p j_p},a_{i_qj_q}\}$  is an isolated pair for all $q$ and $p$. Therefore we have that $\{a_{i_1 j_1}\cdots a_{i_{s}j_{s}}\}$ is an isolated set, as desired.
\end{proof}

The converse of Theorem \ref{thm: monomials=isolated} is false. 
\begin{example}
Consider the matrix 
\[
A = \begin{bmatrix}
    0&1&1&1\\
    1&0&1&1\\
    1&1&0&1\\
    1&1&1&0\\
\end{bmatrix}.
\]
Here $\{a_{12},a_{23},a_{31}\}$ is an isolated set but the corresponding determinant is $x_{12}x_{23}x_{31}+x_{13}x_{21}x_{32}$ and so $x_{12}x_{23}x_{31}\not\in I_3(A)$.
\end{example}

\Cref{lem: 2x2 minors}  allows us to find the generators of the Stanley-Reisner ideal of $\Delta_A$. 
\begin{corollary}
 Utilizing the notation in \eqref{eq: SR ideal}, we have
\[
I_{\Delta_A}= \big (x_{ij}x_{k\ell}: a_{ij}=a_{k\ell}=1, (i,j)\neq (k,\ell), x_{ij}x_{k\ell}\not \in  I_2(A)\big ).
\]
\end{corollary}

Based on this, we introduce a new ideal, which we term the isolation ideal of a Boolean matrix.

\begin{definition}
The \textit{isolation ideal}, $J_A$, of an $m\times n$ Boolean matrix $A$ is the ideal of $k[A]$ defined by 
\begin{equation}\label{eq: J}
J_A = I_{\Delta_A} + \left(x_{ij}^2: 1\leq i\leq m,1\leq j\leq n, a_{ij}=1 \right).
\end{equation}
\end{definition}

We now arrive at the desired algebraic characterization of isolation number as the Castelnuovo-Mumford regularity of the quotient ring defined by the isolation ideal. This invariant also recovers the dimension of the simplicial complex $\Delta_A$.

\begin{theorem}\label{thm: isolation number = reg}
Let $A$ be a Boolean matrix and let $S=k[A]/J_A$ be the quotient ring of the isolation ideal. Then 
\[
\reg(S) =  \iota(A) = \dim(\Delta_A) +1 \leq \brank(A).
\]
\end{theorem}

To prove this characterization, we need the following results.

\begin{lemma}[{\cite[Theorem 18.4]{Peeva}}]\label{lem: reg=top degree}
If $S$ is an Artinian quotient of a polynomial ring (that is, $\dim_k(S)$ is finite) then $\reg(S)=\max\{i\mid S_i\neq 0\}$.
\end{lemma}

With $S$ as in the above lemma, note that $S=\bigoplus_{i\in \mathbb{N}} S_i$ is a graded module where $S_i$ denotes the $k$-vector space spanned by all monomials of degree $i$. The number $\max\{i\mid S_i\neq 0\}$ is called the \textit{top degree} of $S$. 

\begin{lemma}\label{lem: top socle=top isolated}
Let $A$ be a Boolean matrix with isolation ideal $J_A$. A square-free monomial $m=x_{i_1j_1}\cdots x_{i_sj_s}$ is a nonzero element of $S=k[A]/J_A$  if and only if $\{a_{i_1j_1},\ldots, a_{i_sj_s}\}$ is an isolated set in $A$. Moreover $m$ is of top degree in $S$ if and only if the corresponding isolated set has cardinality $s=\iota(A)$. 
\end{lemma}
\begin{proof}
We have that  $\{a_{i_1j_1},\ldots, a_{i_sj_s}\}$ is an isolated set in $A$ if and only if  $m=x_{i_1j_1}\cdots x_{i_sj_s}\not \in I_{\Delta_A}$ by \eqref{eq: SR ideal}. For a square-free monomial $m$, $m\not \in I_{\Delta_A}$ is equivalent to $m\not \in J_A$, which is equivalent to $m\neq 0$ in $S$. Thus we have proven that the nonzero square-free monomials in $S$ correspond  to isolated sets of $A$. Via this correspondence the degree of the monomial $m$ is equal to the cardinality of the isolated set it represents. Since the top degree of $S$ is also the largest degree of a nonzero monomial in $S$ and nonzero monomials in $S$ are square-free, it follows from this correspondence that the nonzero square-free monomials of $S$ of top degree correspond to the isolated sets of largest cardinality and that this largest cardinality is equal to the top degree of $S$. The equality $s=\iota(A)$ now follows from the Definition \ref{def: isolated} of $\iota(A)$.



\end{proof}

We now turn to the proof of Theorem \ref{thm: isolation number = reg}, which relates $\iota(A)$ to the regularity of the quotient ring by the isolation ideal.

\begin{proof}[Proof of Theorem \ref{thm: isolation number = reg}]
 
From \Cref{lem: reg=top degree} we have that
\[\reg(S)=\max\{i\mid S_i\neq 0\}\]
and from \Cref{lem: top socle=top isolated} it follows that 
\[
\iota(A)=\max\{i\mid S_i\neq 0\}.
\]
Moreover $\iota(A)$ is by definition the largest cardinality of a face of $\Delta_A$, while $\dim(\Delta_A)$ is by Definition \ref{def: Delta} one less than the latter number.
From these considerations the conclusion $\reg(S)=\iota(A)=\dim(\Delta_A)+1$ follows. The remaining inequality is derived from Proposition \ref{prop: iso-brank}.
\end{proof}
\section{Examples and Computations}\label{s: computations}

In Sections 3 and 4, we showed two ways to estimate Boolean rank and for each we found an algebraic counterpart based on Castelnuovo-Mumford regularity. In this section, we provide examples showing the limitations and sharpness of these estimates.
\subsection{Block Matrices}
\begin{definition}
A \textit{block diagonal matrix} is a matrix of the form 
\[
\begin{bmatrix}
    B_1 & & 0\\
    & \ddots &\\
    0& & B_r
\end{bmatrix}
\]
where $B_1, \ldots, B_r$ are matrices lying along the diagonal and all the other entries of the matrix equal 0.

We define a \textit{solid block diagonal matrix} to be a block diagonal matrix where each $B_i$ is a matrix of all 1's.
\end{definition}
\tikzset{every path/.append style = { draw, dashed, rounded corners } }

\begin{example}
(A) is a block diagonal matrix, but not a solid block diagonal matrix. (B) is an example of a solid block diagonal matrix. 
\begin{center}
\begin{tabular}{ c c c }

$\begin{bNiceMatrix}[margin=2pt] 
\Block[tikz={red}]{2-2}{}1 & 0 & 0 & 0\\
1 & 1 & 0 & 0\\
0 & 0 & \Block[tikz={red}]{2-2}{}1 & 0 \\
0 & 0 & 0 & 1 
\end{bNiceMatrix}$ &$\begin{bNiceMatrix}[margin=2pt] 
\Block[tikz={red}]{2-3}{}1 & 1 & 1 & 0 & 0 & 0\\
1 & 1 & 1 & 0 & 0 & 0\\
0 & 0 & 0 &\Block[tikz={red}]{2-3}{} 1 & 1 & 1 \\
0 & 0 & 0 & 1 & 1 & 1
\end{bNiceMatrix}$\\ 
 (A) & (B)  
\end{tabular}
\end{center}

\end{example}

Block matrices are the most favorable of the types of matrices to compute Boolean matrix rank since the data 
is already neatly organized into groups. Theorem \ref{thm: block diagonal} below demonstrates that the isolation number and the regularity are sharp estimates for the Boolean rank.

\begin{theorem}
    \label{thm: block diagonal}
Let $A$ be an $m\times n$ Boolean solid block diagonal matrix and let  $G$ be the corresponding bipartite graph. Then \begin{enumerate}
    \item The regularity of $R/I_G$ is the number of blocks.
    \item The Boolean rank of $A$ is the number of blocks.
    \item The isolation number of $A$ is the number of blocks.
\end{enumerate}
\end{theorem}
\begin{proof}
(1) The edge ideal of $G$ is a sum of edge ideals corresponding to each block $B_i$
\[
I_G=I_{B_1}+\cdots+I_{B_r}.
\]
As the ideals $I_{B_1},\ldots, I_{B_r}$ involve pairwise disjoint sets of variables, one deduces 
\begin{equation}\label{eq: sum of reg}
\reg(R/I_G)=\reg(R/I_{B_1})+\cdots+\reg(R/I_{B_r}).
\end{equation}
Since the bipartite graph with biadjacency matrix $B_i$ is a biclique we know by Lemma \ref{lem: regularity biclique} that $\reg(R/I_{B_i})=1$ for each solid block and hence the sum in \eqref{eq: sum of reg} is equal to the number $r$ of blocks.

(2) Each block corresponds to a biclique, so the number of bicliques necessary  to cover the graph is at most the number of blocks. That is, $\brank(A)={\rm bc}(G)\leq r$. 
By Theorem \ref{thm: reg} and part (1) we know that $r=\reg(R/I_G)\leq \brank(A)$. 
  Hence the Boolean rank must be the number of blocks of $A$.

(3) 
 We claim that the set formed by the entry 1 in the top right corner of each block is isolated. Indeed, fixing any two such 1's, the entry in the lower block has only zeros in its column above it. In particular, there is a zero entry in the upper right of the $2\times 2$ minor that contains the two fixed 1's. This gives us a set of isolated ones with size that is at least the number of blocks in the matrix $A$. We have shown $\iota(A)\geq r$.
However, by Proposition \ref{prop: iso-brank} and part (2)  we know that $\iota(A)\leq \brank(A)=r$. So we conclude that the isolation number must be exactly the number of blocks.
\end{proof}

\begin{example}
Take
\[
A=\begin{bNiceMatrix}[margin=2pt]
  \Block[tikz={red}]{2-3}{}1&1&\textcolor{magenta}{\bf 1}&0&0\\
  1&1&1&0&0\\
  0&0&0&\Block[tikz={red}]{2-2}{}1&\textcolor{magenta}{\bf 1}\\
  0&0&0&1&1
\end{bNiceMatrix}
\]
This matrix has $\reg(R/I_G)=\iota(A)=\brank(A)=2$. 
\end{example}
\subsection{Overlapping Block Matrices}
In this section we look at a broader family of block diagonal matrices. We call these matrices \textit{overlapping block matrices} and describe them as follows. 

\begin{definition}
An \textit{overlapping block diagonal matrix} is a matrix $A$ of the form 
\[
A=\begin{bmatrix}
    B_1 & & 0\\
    & \ddots &\\
    0& & B_r
\end{bmatrix}
\]
where all entries of $B_1, \ldots, B_r$ are 1 and all the other entries of $A$ equal 0. In particular, each block $B_i$ need not have disjoint rows and/or columns from either block $B_{i-1}$ or $B_{i+1}$.

A Boolean matrix can be considered an overlapping block matrix in multiple ways. Thus when considering an overlapping block matrix we mean the matrix together with a choice of blocks as described above. We indicate the choice of blocks in examples by dashed borders.
\end{definition}

The family of overlapping block diagonal matrices is extremely broad. In order to make precise statements about Boolean rank, regularity and isolation number, we must place some additional constraints.

\begin{definition}

A overlapping block diagonal matrix $A$ is \textit{column separated} if every block has a column that does not belong to another block.

A overlapping block diagonal matrix $A$ is \textit{row separated} if every block has a row that does not belong to any other block.
\end{definition}
\begin{example}
(A) is an example of a matrix that is not column or row separated. (B) is an example of a column separated matrix, but not row separated. (C) is an example of a row separated matrix, but not column separated.
\begin{center}
\begin{tabular}{ c c c }
 $\begin{bNiceMatrix}[margin=2pt] 
\Block[tikz={red}]{2-2}{}1 & 1 & 0 & 0\\
1 &  \Block[tikz={red}]{2-2}{}1 & 1 & 0\\
0 & 1 &  \Block[tikz={red}]{2-2}{}1 & 1\\
0 & 0 & 1 & 1 
\end{bNiceMatrix}$ & $\begin{bNiceMatrix}[margin=2pt]
 \Block[tikz={red}]{2-3}{}1&1&1&0&0&0\\
  1&1& \Block[tikz={red}]{2-3}{}1&1&1&0\\
  0&0&1&1& \Block[tikz={red}]{2-2}{}1&1\\
  0&0&0&0&1&1
 \end{bNiceMatrix}$ & $\begin{bNiceMatrix}[margin=2pt]
   \Block[tikz={red}]{2-2}{}1&1&0&0\\
  1& \Block[tikz={red}]{3-2}{}1&1&0\\
  0&1&1&0\\
  0&1& \Block[tikz={red}]{2-2}{}1&1\\
  0&0&1&1
 \end{bNiceMatrix}$ \\ 
 (A) & (B) & (C)  
\end{tabular}
\end{center}
\end{example}

The following is a generalization of Theorem \ref{thm: block diagonal}, thus enlarging the class of matrices for which the isolation number and the regularity are sharp estimates for the Boolean rank.

\begin{theorem}
Let $A$ be an $m\times n$ Boolean overlapping block diagonal matrix and let  $G$ be the corresponding bipartite graph. Further suppose that $A$ is  column separated and row separated. Then \begin{enumerate}
    \item The regularity of $R/I_G$ is the number of blocks.
    \item The Boolean rank of $A$ is the number of blocks.
    \item The isolation number of $A$ is the number of blocks.
\end{enumerate}
\end{theorem}
\begin{proof}
(1) We know that each block corresponds to a biclique. So the edge ideal of $A$ is a sum of edge ideals corresponding to each block $B_i$:
\[
I_G=I_{B_1}+\cdots+I_{B_r}.
\]
By Corollary \ref{2.9}, we know that 
\begin{equation}\label{eq:1B}
\reg(R/I_G)= \reg(R/(I_{B_1}+\ldots + I_{B_r}))\leq \sum_{i=1}^r \reg (R/I_{B_i}) = r,
\end{equation}
where $r$ is the number of blocks in $A$.
For the other inequality, recall from Lemma \ref{lem:indmatch} that 
\[
\indmatch(G) \leq \reg(R/I_G).
\]
Note that an induced matching of a bipartite graph corresponds to an identity submatrix  of the biadjacency matrix of the bipartite graph. Hence the induced matching number of $G$ is the largest size of an identity submatrix found in the biadjacency matrix of the graph, $A(G)$. For overlapping block diagonal matrices that are both column separated and row separated, one can find an identity submatrix by choosing an entry of 1 from each block that is in the independent rows and columns of each block. By the above considerations and Lemma \ref{lem:indmatch} we have 
\begin{equation}\label{eq: 2}
 \reg(R/I_G) \geq \indmatch(G)  \geq \text{number of blocks}.
\end{equation}

By \eqref{eq:1B} and \eqref{eq: 2} we have the regularity of $R/I_G$ is exactly the number of blocks, as desired.

(2) 
Putting together \eqref{eq: inequalities} and the fact that the $r$ blocks form a biclique cover of $G$, we have the sequence of inequalities
\[
    \indmatch(G) \leq \reg(R/I_G) \leq  \brank(A) = \text{bc}(G)\leq r.
\]

As shown in the proof of (1), 
we have $\indmatch(G) \geq r$, yielding equality throughout.
In particular we have shown
\[
\text{number of blocks in $A$}=  \brank(A).
\]

(3) The proof is similar to part (3) of Theorem \ref{thm: block diagonal}. The crucial hypothesis allowing to find an isolated set of size $r$ as in that proof is that $A$ is both row and column separated.
\end{proof}
\begin{example}
Take
\[
A=\begin{bNiceMatrix}[margin=2pt]
  \Block[tikz={red}]{2-2}{} 1&1&0&0 &0\\
  1& \Block[tikz={red}]{3-3}{}1&1& 1& 0\\
  0&1&1& 1&0\\
  0&1& 1&\Block[tikz={red}]{2-2}{}1&1\\
  0&0& 0& 1&1
 \end{bNiceMatrix}
\]
This matrix has $\reg(R/I_G)=\iota(A)=\brank(A)=3$. 
\end{example}

\subsection{Identity Complement Matrices}
The last family of matrices that we compute the estimates for is what we call \textit{identity complement matrices}. 
This family of matrices has been considered in \cite{deCaen}.

\begin{definition}
Define the  \textit{identity complement matrix} denoted $J$ to be the square matrix with entries 
\[
J_{ij}=\begin{cases} 
1 & \text{ if } i\neq j\\
0 & \text{ if } i= j.
\end{cases}
\]
\end{definition}


\begin{theorem}\label{thm: identity complement}
Let $A$ be an $n\times n$ Boolean identity complement matrix  with $n\geq 3$ and let $G$ be the corresponding bipartite graph. Then \begin{enumerate}
    \item The regularity of  $R/I_G$ is 2.

    \item The isolation number of $A$ is 3.
\end{enumerate}
\end{theorem}

However, it is proven in \cite[Corollary 2]{deCaen} that the Boolean rank of $A$ is min$\left \{ k : n \leq \displaystyle\binom{k}{\lfloor \frac{k}{2}\rfloor} \right \}$.

To prove Theorem \ref{thm: identity complement}(1), first we must establish a useful tool, called  the \textit{Ferrers graph}.
\begin{definition}
A \textit{Ferrers graph} is a bipartite graph $G=(X\sqcup Y,E)$ on vertex sets $X = \{x_1,\ldots,x_n\}$ and $Y = \{y_1,\ldots,y_m\}$ such that  $\{x_1,y_m\}$ and $\{x_n,y_1\}$ are required to be edges of $G$ and if $\{x_i,y_j\}$ is an edge of $G$, then so is $\{x_p,y_q\}$ for $1\leq p\leq i$ and $1\leq q\leq j$.
\end{definition}

  The name is a reference to the fact that a biadjacency matrix $A$ represents a Ferrers graph if the nonzero entries of $A$ form a \textit{Ferrers tableau}.

Corso and Nagel \cite{NagelCorso} have studied the algebraic properties of the edge ideal $I = I_G$ associated to a Ferrers graph $G$.
 
\begin{theorem}\cite[Thm.\,4.2]{NagelCorso}\label{thm:FerrersIdealReg}
Let $G$ be a bipartite graph without isolated vertices. Then its edge ideal, $I_G$, has regularity 2 (equivalently, $\reg(R/I_G)=1$)  if and only if $G$ is a Ferrers graph.
\end{theorem}

A particular case appears in Lemma \ref{lem: regularity biclique} since bicliques are a type of Ferrers graph that corresponds to a rectangular Ferrers tableau.

\begin{proof}[Proof of Theorem 5.11]
(1) An identity complement matrix is comprised of two Ferrers diagrams, outlined in  an example.
\[\begin{tikzpicture}

\matrix[matrix of math nodes,left delimiter = {[},right delimiter = {]},row sep=10pt,column sep = 10pt] (m)
 {
  0 & 1 & 1 & 1 \\
  1 & 0 & 1  & 1 \\
  1 & 1 & 0  & 1 \\
  1 & 1 & 1  & 0 \\
 };

\begin{pgfonlayer}{background}
 \node[inner sep=3pt,fit=(m-1-2)]          (1)   {};
 \node[inner sep=3pt,fit=(m-1-3) (m-2-4)]  (2)   {};
 \node[inner sep=3pt,fit=(m-2-3) (m-2-4)]  (3)   {};
 \node[inner sep=3pt,fit=(m-3-4)]          (4)   {};
 \draw[red, rounded corners,dashed,inner sep=3pt,fill opacity=0.1] (1.north west) -- (2.north east) |- (4.south west)|- (3.south west) |- (2.south west) |- (1.south west) -- cycle;
\end{pgfonlayer}

\begin{pgfonlayer}{background}
 \node[inner sep=3pt,fit=(m-2-1)]          (1)   {};
 \node[inner sep=3pt,fit=(m-3-1) (m-2-2)]  (2)   {};
 \node[inner sep=3pt,fit=(m-3-2)]  (3)   {};
 \node[inner sep=3pt,fit=(m-4-3)]          (4)   {};
 \draw[red, rounded corners,dashed,inner sep=3pt,fill opacity=0.1] (1.north east) -- (2.north west) |- (4.south east)|- (3.south east) |- (2.south east) |- (1.south east) -- cycle;
\end{pgfonlayer}
\end{tikzpicture}\]
Thus we have $I_G= F_1 + F_2$ for Ferrers ideals $F_i$. By Proposition \ref{thm: reg sum} and Theorem \ref{thm:FerrersIdealReg}, we have that 
\[
\reg(R/I_G)\leq \reg(R/F_1)+\reg(R/F_2)=1+1=2.
\]

For the other direction, we will make use of Lemma \ref{lem:indmatch}. Observe that for any $i\neq j$ the edges $\{x_i,y_j\}$ and $\{x_j,y_i\}$ form an induced matching as rows $i,j$ and columns $i,j$ of $J$ form an identity submatrix.   Hence we obtain $2\leq \text{indmatch}(G)\leq \reg(R/I_G)$ which together with the previous arguments gives $\reg(R/I_G)=2$, as desired.

(2) First we show that $\iota(A)\geq 3$. Take any three distinct indices $i,j,k$. Consider the three entries $a_{ij}, a_{jk},$ and $a_{ki}$. Because $i,j,k$ are distinct, we know that $a_{jj}=a_{kk}=a_{ii}=0$. Therefore $a_{ik}a_{jj}=0, a_{ji}a_{kk}=0,$ and $a_{kj}a_{ii}=0$. Thus $a_{ij}, a_{jk},$ and $a_{ki}$ are pairwise isolated from one another. Hence there is an isolated set of size 3, so 
$\iota(A)\geq 3.$ Now we will show that there cannot be an isolated set of size 4.

Suppose $a_{ij} = 1$.
In order for some other entry $a_{k\ell} = 1$ to be isolated from $a_{ij},$ we must have $\ell = i$ or $k = j$, as otherwise $a_{i\ell} = 1 = a_{kj}$ from which it follows that $a_{i\ell}\wedge a_{kj} = 1$ and so $a_{ij}, a_{k\ell}$ are not isolated.
As pairwise isolated ones cannot be in the same row or column, there is at most one entry in row $j$ and at most one entry in column $i$ isolated from $a_{ij}$.
Thus, there are at most two other ones in $A$ isolated from any entry $a_{ij} =1$, and so $\iota(A) \leq 3$, completing the proof.
\end{proof}
 
\begin{example}\label{ex: anti id example}
Consider
\[
A=\begin{bmatrix}
    0&1&1&1&1&1\\
    1&0&1&1&1&1\\
    1&1&0&1&1&1\\
    1&1&1&0&1&1\\
    1&1&1&1&0&1\\
    1&1&1&1&1&0
\end{bmatrix}. 
\]
This matrix has $\reg(R/I_G)=2$, $\iota(A)=3$ and $\brank(A)=4$, showing that both $\reg(R/I_G)$ and $\iota(A)$ as lower bounds for the Boolean rank need not be tight. 

In fact, denoting by $J_n$ the $n\times n$ identity complement matrix, one sees from Theorem \ref{thm: identity complement} that the  Boolean rank of $J_n$ grows unboundedly in $n$, while the corresponding regularity and isolation numbers stay constant. This shows that the isolation number and the regularity can be arbitrarily far from the Boolean rank.
\end{example}

\section{Conclusion}
The gap between regularity and Boolean rank in \Cref{thm: identity complement} stems from the result of Corso and Nagel (Theorem~\ref{thm:FerrersIdealReg}), which suggests a refinement of Theorem~\ref{thm: reg}. Denote by $\fc(G)$ the \emph{Ferrers cover number} of a bipartite graph $G$, defined analogously to the biclique cover number $\bc(G)$ (Definition~\ref{def:biclique}), as the smallest number of Ferrers graphs which can be used to cover all edges of $G$. As bicliques are Ferrers graphs, we get the inequality $\fc(G) \leq \bc(G)$ for a bipartite graph $G$.
The proof of Theorem~\ref{thm: reg}, \emph{mutatis mutandis}, (i.e., employing Theorem~\ref{thm:FerrersIdealReg} instead of Lemma \ref{lem: regularity biclique}) then shows that 
\begin{equation}\label{eq: ferrers inequality}
    \reg(R/I_G) \leq \fc(G).
\end{equation}

This leads to the interesting problem of estimating how tight the inequality above bounds the Ferrers cover number,
which we propose for further investigation.

\begin{question}
    Can the difference between the two invariants in \eqref{eq: ferrers inequality} be arbitrarily large when $G$ ranges among all bipartite graphs?
\end{question}

\bibliographystyle{amsalpha}
\bibliography{biblio}

\end{document}